%% file: main.tex
\documentclass[psamsfonts]{amsart}

\input{arxiv_preamble}

\title{Shirshov's amalgamated free product and generic nilpotent groups}

\thanks{$^\dagger$ supported by the Basque Government grant reference number IT1913-26, and by the Ramon y Cajal grant RYC2023-042677-I funded by MICIU/AEI/10.13039/501100011033 and by ESF+. $^\ddagger$ supported by NSF grant DMS-2246992 and NSF CAREER award DMS-2442011.}

\author[C. D'Elb\'{e}e]{Christian D'Elb\'{e}e$^{\dagger}$}

\address{University of the Basque Country, Department of Mathematics (Leioa) / Institute for Logic, Cognition, Language and Information (Donostia-San Sebastián), Spain;}
\email{christian.delbee@ehu.eus} 
\urladdr{\href{http://choum.net/\textasciitilde chris/page\textunderscore perso/}{http://choum.net/\textasciitilde chris/page\textunderscore perso/}}

\author[I. M\"uller]{Isabel M\"uller}
\address{Department of Mathematics and Actuarial Science \\
The American University in Cairo \\ Egypt }
\email{isabel.muller@aucegypt.edu}
\urladdr{\href{https://www.aucegypt.edu/fac/isabel}{https://www.aucegypt.edu/fac/isabel}}

\author[N. Ramsey]{Nicholas Ramsey$^\ddagger$}
\address{Department of Mathematics \\
University of Notre Dame\\
 USA}
\email{sramsey5@nd.edu}
\urladdr{\href{https://math.nd.edu/people/faculty/nicholas-ramsey/}{https://math.nd.edu/people/faculty/nicholas-ramsey/}}

\author[D. Siniora]{Daoud Siniora}
\address{Department of Mathematics and Actuarial Science \\The American University in Cairo\\
 Egypt}
\email{daoud.siniora@aucegypt.edu}
\urladdr{\href{https://sites.google.com/view/daoudsiniora/}{https://sites.google.com/view/daoudsiniora/}}

\date{\today}

\begin{document}

\maketitle

\begin{abstract}
In earlier work, the authors gave a construction and description of an amalgamated free product of filtered Lie algebras within a fixed nilpotency class, based on an intricate induction rooted in the work of Maier and of Higman. In this paper, the authors give a new description of this amalgam, adopting the viewpoint and methods from the work of A. I. Shirshov on amalgamation of Lie algebras, substantially simplifying their previous approach. This new description yields both group-theoretic and descriptive set-theoretic applications. A $c$-nilpotent group is called UL-equivalent if its lower and upper central series coincide. We prove that every finite $c$-nilpotent group of prime exponent $p$ with $p>c$ embeds in a finite UL-equivalent $c$-nilpotent group of exponent $p$. This recovers a result due to Ivanov and Majcher that the Polish space of enumerated $c$-nilpotent groups of exponent $p>c$ has a comeager orbit. Our result also has the following consequences. First, the class of finite $c$-nilpotent groups of exponent $p>c$ has the \textit{cofinal} amalgamation property (answering a question of Ivanov and Majcher, who showed that it has the \textit{weak} amalgamation property). Second, the reduct of the Fraïssé limit of $c$-Lazard groups of exponent $p>c$ to the group language is generic in the space of enumerated groups. Finally, we prove analogues of these results also for torsion-free $c$-nilpotent groups.
\end{abstract}

\setcounter{tocdepth}{1}
\tableofcontents

\section{Introduction}

This article is the third in a series by the authors, investigating the properties of generic nilpotent groups. Although from a model-theoretic point of view, abelian groups are quite tame, the moment one takes a tiny step into the non-abelian world, one finds arbitrarily complicated groups. For example, constructions of Mekler and related ones due to Ershov yield that even $2$-nilpotent groups of exponent $p$ (for an odd prime $p$) can code arbitrary graphs. This leads naturally to the question: is the complexity that one can find in the class of nilpotent groups typical or is it special? If one picks a nilpotent group at random, should one expect that it be tame or wild? Of course, for this to be a proper mathematical question, one has to specify what is meant by a `typical', `random', or `generic' nilpotent group, and how one will measure tameness and complexity. In \cite{DINCI} and \cite{DINCII}, we interpreted these questions in the way model theorists tend to interpret them, by investigating existentially closed nilpotent groups via the model-theoretic properties of their model companions and determining their place in neostability's map of the universe. Another, classical approach for typicality in groups is due to Gromov and its models of randomness \cite{gromov1993}, which seems also suitable for nilpotent groups, see \cite{cordes2018random}. A third and important approach to genericity stems from descriptive set theory, in the study of the space of \textit{enumerated groups} \cite{goldbring2023generic, ivanov1999genericity,truss1992}. This is where the main applications of our present results will occur. Before getting to those applications, let us describe the starting point of the present paper and its main result.

The core of the authors' work in \cite{DINCI} was the construction of an amalgamated free product of filtered Lie algebras within a fixed nilpotency class, which we called \emph{Lazard Lie algebras} (see Definition \ref{def:LLA}). The construction of this amalgam constituted a substantial part of the paper, built on a intricate induction rooted in the work of Maier \cite{Maierexpp}, and Higman \cite{higman64} before that. It leveraged a notion of rank and a set of tools (what we called the Malcev basis and Malcev ``calculus") that allowed us to inductively construct the amalgam and analyze its essential features. In the present paper, we revisit this construction and give it a new description following the classical approach of Shirshov on free Lie algebras. While working on \cite{DINCI}, it was clear to the authors that, in some particularly easy cases, a basis of the amalgam of Lazard Lie algebras could be described in terms of chosen bases for the Lazard Lie algebras being amalgamated, roughly in terms of Hall monomials. This was canonical in the sense that given the data of the amalgamation problem and suitable bases for the factors, one could mechanically produce an explicit list of monomials which would constitute a basis of the amalgam. However, in full generality, our tools seemed inadequate to give this explicit basis in simple terms. This was due to our description of the amalgam, which was essentially characterized by its inductive definition. This is where Shirshov's approach via combinatorics of words (regular/special words, so-called Lyndon-Shirshov words) turned out to be extremely useful and, in conjunction with some of the tools developed in \cite{DINCI}, provided us with the right viewpoint to obtain a new description of this amalgam, circumventing the induction and thus simplifying drastically the construction. It turned out that this result has many applications beyond the theory of Lie algebras.

The first application is a purely group-theoretic result. If $G$ is a nilpotent group, then it is known that the upper central series $(Z_i(G))$ and lower central series $(\gamma_i(G)))$ have the same length, though the subgroups that appear in these two series are not the same in general. A group $G$ of nilpotency class $c$ is called \textit{UL-equivalent} if for all $1\leq i\leq c+1$ we have $\gamma_i(G) = Z_{c+1-i}(G)$\cite{grouppropsULequivalentGroup}. Examples of such groups include finite $p$-groups of maximal class. Using our new description of the amalgam and the Lazard correspondence, we obtain a positive answer to a question asked by Ivanov and Majcher \cite[Bottom of page 12]{ivanov2025genericgroupsweakamalgamation}:

\begin{question}
    Is it possible to embed any finite $c$-nilpotent group of exponent $p$ with $p>c$ in a finite UL-equivalent $c$-nilpotent group of exponent $p$?
\end{question}

We prove that every group of nilpotency class $c$ and exponent $p$ with $p>c$ embeds in a UL-equivalent group of nilpotency class $c$ and exponent $p$. Further, if the starting group is finite, the UL-equivalent group can also be chosen to be finite, answering the question.

This rather innocuous group-theoretic result has an unexpected consequence in the study of enumerated groups, as we describe now. If $\mathbb{K}$ is a countable class of finitely generated structures, the collection of non-finitely generated structures with underlying set is $\omega$ and whose age is contained in $\mathbb{K}$ can be organized into a topological space, in which the set of models containing a given finitely generated structure, with underlying set contained in $\omega$, as a substructure is declared to be a basic open set. In this setting, the natural notion of a generic property is one that holds on a comeager set of models. Specializing for $T$ being the theory of groups, we obtain the space $\mathcal{G}$ of \textit{enumerated groups}, and for any group-theoretic property $P$, the subspace $\mathcal{G}_P$ of groups satisfying $P$. In \cite{goldbring2023generic} the question was raised whether a comeager subset of $\mathcal{G}_P$ could consist of isomorphic groups  (Question 1.0.1 (2)). Their work proves that this is generally hard to predict, and that, for instance, $\mathcal{G}$ does not admit such a comeager subset. On the other hand, for $P$ being the property of being nilpotent of class at most $c\in \N$ and of prime exponent $p>c$, Ivanov and Majcher \cite{ivanov2025genericgroupsweakamalgamation} proved that $\mathcal{G}_{c,p} := \mathcal{G}_P$ admits a comeager isomorphism class, answering positively the question above. 
Using the machinery developed by Ivanov in \cite{ivanov1999genericity}, they obtain their result by proving that the class $\mathbb{G}_{c,p}$ of finite $c$-nilpotent groups of exponent $p>c$ satisfies a weakening of the amalgamation property (AP) called the \textit{weak amalgamation property} (WAP), see Definition \ref{def:APs}. This notion, introduced by Ivanov \footnote{It was introduced under the name \textit{almost amalgamation property}.} in \cite{ivanov1999genericity}, also occurs in the work of Kechris and Rosendal \cite{kechrisRosendal2007} on generic automorphisms of homogeneous structures. Sitting strictly between (AP) and (WAP) is the notion of \textit{cofinal amalgamation property} (CAP, Definition \ref{def:APs}) which was fist introduced by Calais \cite{calais1967} then rediscovered by Truss in  his fundamental paper \cite{truss1992}. Summarizing, we have:

\[\text{(AP)}\implies \text{(CAP)}\implies \text{(WAP)}\]It is easy to see that the class $\mathbb{G}_{c,p}$ fails (AP), and our result yields that $\mathbb{G}_{c,p}$ satisfies the (CAP), using that UL-equivalent groups are in fact amalgamation bases.

Finally, another consequence of our result, pointed out by Ivanov and Majcher in \cite[Remark 2.12]{ivanov2025genericgroupsweakamalgamation}, is an explicit example of a \textit{generic group}, in the sense of an element of $\mathcal{G}_{c,p}$ whose isomorphism class is comeagre in $\mathcal{G}_{c,p}$.
The failure of (AP) for the class $\mathbb{G}_{c,p}$ can be in a sense compensated by expanding the language. Let $\mathbb{G}_{c,p}^P$ be the class of finite \textit{Lazard} $c$-nilpotent groups of exponent $p>c$, i.e. groups $G\in \mathbb{G}_{c,p}$ equipped with a filtration of subgroups $P_1(G)\supseteq\ldots \supseteq P_{c+1}(G) = 1$ satisfying the property $[P_i,P_j]\seq P_{i+j}$. The class $\mathbb{G}_{c,p}^P$ was introduced by Baudisch \cite{Baudisch2}, building on the work of Maier \cite{Maierexpp} and revisited recently by the authors in \cite{DINCI}. Unlike $\mathbb{G}_{c,p}$, the class $\mathbb{G}_{c,p}^P$ satisfies (AP) and is a Fraïssé class. Its Fraïssé limit $\mathbf{G}_{c,p}^P$ is $\omega$-categorical. Using a result of Kabluchko and Tent \cite{kabluchko2021universalhomogeneousstructuresgeneric}, the isomorphism class of $\mathbf{G}_{c,p}^P$ is generic in the corresponding space. The question of whether the reduct $\mathbf{G}_{c,p}$ of $\mathbf{G}_{c,p}^P$ to the pure group language is generic in $\mathcal G _{c,p}$ is nontrivial and follows from our group-theoretic result, using \cite[Remark 2.12]{ivanov2025genericgroupsweakamalgamation}. Note that $\mathbb{G}_{c,p}$ is a weak Fraïssé class \cite{kruckmanthesis2016,kabluchko2022weakfraisselimits}, hence admits a weak Fraïssé limit, which is generic in the corresponding class by \cite{kabluchko2022weakfraisselimits}. 

Finally, our approach allows us to prove analogous theorems for the case of torsion-free $c$-nilpotent groups. Since the results of \cite{DINCI} concerning Lie algebras did not have any hypothesis about the underlying field, we know the class of $c$-nilpotent Lazard Lie algebras over $\mathbb{Q}$ also forms a Fra\"iss\'e class with free amalgamation. As a consequence, our methods allow us also to show that the class of torsion-free $c$-nilpotent groups has a generic group, and we show that this group is the reduct to the language of groups of the group corresponding, via the Malcev correspondence, to the Fra\"iss\'e limit of finitely generated $c$-nilpotent Lazard Lie algebras over $\mathbb{Q}$.




\section{Preliminaries}

We will recall the different notions of amalgamation.

\begin{definition}\label{def:APs}
    
We say that a class $\CC$ of structures has
\begin{enumerate}
    \item the \textit{amalgamation property} if for all $A,B,C\in \mathscr{C}$ with embeddings $f_0:C\to A$, $g_0:C\to B$ there exists $D\in \mathscr{C}$ and embeddings $f_1 : A\to D$, $g_1:B\to D$ such that $f_1\circ f_0\upharpoonright C = g_1\circ g_0\upharpoonright C$;
    \item the \textit{cofinal amalgamation property} if there is a subclass $\CC'\subseteq \CC$ which has the amalgamation property and such that any $A\in \CC$ embeds into some $A'\in \CC'$;
     \item the \textit{weak amalgamation property} (WAP) if for all $C\in \mathscr{C}$ there is some $C’\in \mathscr{C}$ together with an embedding $h:C\rightarrow C’$, such that for all $A,B \in \mathscr{C}$ with embeddings $f_0:C’\to A$, $g_0:C’\to B$ there exists $D\in \mathscr{C}$ and embeddings $f_1 : A\to D$, $g_1:B\to D$ such that $f_1\circ f_0\circ h\upharpoonright C = g_1\circ g_0\circ h \upharpoonright C$.
\end{enumerate}
\end{definition}
\begin{definition}\label{def:LLA}
    Given a natural number $c\in \N$, a $c$-\textit{nilpotent} \textit{Lazard Lie algebra} ($c$-LLA, or LLA when the $c$ is implicit) $L$ over a field $\F$ is a Lie algebra $L$ over $\F$ equipped with a chain of distinguished subalgebras $L = L_{1} \geq L_{2} \geq \cdots \geq L_{c+1} = 0$ satisfying the following property
    \begin{itemize}\itemsep0.7em
        \item $[L_{i},L_{j}] \leq L_{i+j}$ for all $i,j$, where $L_k = 0$ for all $k > c$.
    \end{itemize}
\end{definition}

In \cite{DINCI}, we showed that the class of $c$-LLAs has the free amalgmation property in the sense below. 

\begin{definition}[Baudisch]\label{def:freeamalgambaudisch}
     Let $A,B,C$ be LLAs with embeddings $f_0 :C\to A$ and  $g_0:C\to B$. We say that a Lazard Lie algebra $S$ is a \textit{free amalgam} of $A$ and $B$ over $C$ if 
    
    \begin{enumerate}\itemsep0.7em
    \item $S$ is an amalgam of $A$ and $B$ over $C$, i.e., there are embeddings $f_1:A\to S$, $g_1:B\to S$ such that $f_1\circ f_0 = g_1\circ g_0$; 
    \item $S =\langle A'B'\rangle$ where $A' = f_1(A)$ and $B' = g_1(B)$;
    \item \textit{(Strong amalgam)} $A'\cap B' = C'$ where $C' = (f_1\circ f_0)(C)$;
    \item \textit{(Freeness)} for any LLA $D$ and any LLA \textit{homomorphisms} $f:A\to D$ and $g:B\to D$ such that $f\circ f_0(z)=g\circ g_0(z)$ for $z\in C$, there exists a (unique) homomorphism $h : S\to D$ such that $h\circ f_1=f$ and $h\circ g_1 =g$.
    \end{enumerate}
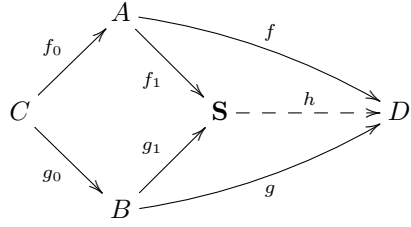
    \[ \xymatrix{ & A \ar[rd]_{f_1}\ar@/^/[rrrd]^{f} \\    
             C\ar[ur]^{f_0} \ar[dr]_{g_0} & & \mathbf{S}\ar@{-->}[rr]^{h} & & D\\   
              & B \ar[ru]^{g_1}\ar@/_/[rrru]_{g}}
\]
\end{definition}

  If such a free amalgam of $A$ and $B$ over $C$ exists, it is unique and we denote it by $A\otimes_C B$.

\section{Shirshov's construction}

Our aim in these notes is to give an alternative description of the free amalgam of $c$-nilpotent Lazard Lie algebras presented in \cite{DINCI}.  Here is the rough idea: Shirshov showed how to construct an amalgamated free product of Lie algebras and, significantly, explained how to obtain a basis for this Lie algebra \cite{shirshov1962hypothesis}. So given $c$-nilpotent Lazard Lie algebras $A$ and $B$ containing a common subalgebra $C$, we forget the LLA structure and form the amalgamated free product $L_{0} = A *_{C} B$ in the category of Lie algebras. The LLA structure on $A$ and $B$ induces an LLA structure on $L_{0}$ and we quotient by the piece of this gradation indexed by $c+1$. This produces a $c$-nilpotent LLA into which $A$ and $B$ embed. Using the properties of Shirshov's basis, we can conclude that this quotient is isomorphic to the free amalgam $A \otimes_{C} B$. 

\subsection{Shirshov's amalgamated free product}

\begin{defn}\cite{shirshov2009subalgebras, shirshov1958free, shirshov1962hypothesis}
    Suppose $X$ is a set totally ordered by $<$. We will regard $X$ as an alphabet. We denote by $X^*$ the set of associative words in the letters $X$, which we endow with the lexicographic order (also denoted $<$). 
    \begin{enumerate}
        \item A word $w \in X^{*}$ is called \emph{regular} if whenever $w$ is factored $w = uv$ where $u$ and $v$ are non-empty words, then $w > vu$. 
        \item We define the \emph{canonical bracketing} of a regular associative word $w$ by induction on the length of $w$. If $w \in X$, then $[w] = w$. If the length of $w$ is greater than $1$, then we can write $w = uv$ where $v$ is the longest proper regular suffix of $w$. By \cite[Lemma 2.15]{BokutChen2007}, then $u$ is also regular and we set $[w] = [[u],[v]]$.
    \end{enumerate}
    Two words (associative or non-associative) are said to have the \emph{same content} if each element of $X$ occurs in each word the same number of times. The \emph{content} of a word, then, can be understood as the multiset of the symbols occuring in the word, taking into account their multiplicity. 
\end{defn}

By abuse of language, we will refer to a canonical bracketing of a regular word as a regular word, when context makes it clear that we are not referring to an associative word. The notion of a regular word was introduced by Shirshov in \cite{shirshov1958free}, where he proved that the canonical bracketings of regular words form a basis of the free Lie algebra generated by $X$. For a more contemporary presentation of this result, see also \cite[Section 5.3]{Khu98} (where canonical bracketings are called \emph{basic products}). 

Now we describe Shirshov's method for constructing an amalgamated free pro\-duct of Lie algebras. While in our set up we only need to amalgamate two Lie algebras, we will present Shirshov's general method for an entire family of Lie algebras. We thus start with an indexed family of Lie algebras $(L_{\alpha})_{\alpha \in I}$, all of which contain a common subalgebra $L_{*}$. We choose a basis $(e_{\beta})_{\beta \in J}$ for $L_{*}$.  For each $\alpha \in I$ and $\beta \in J$, we let $e_{\alpha,\beta} = e_{\beta}$ and we extend $(e_{\alpha,\beta})_{\beta \in J}$ to a basis $(e_{\alpha,\beta})_{\beta \in J_{\alpha}}$ for $L_{\alpha}$ (so $J \subseteq J_{\alpha})$. We may assume that the index sets $J_{\alpha}$ are disjoint over $J$. We put an ordering $<_{I}$ on $I$ and an ordering $<_{J_{\alpha}}$ extending a common ordering $<_{J}$ on $J$ for each $\alpha \in I$ with the property that $\beta \in J$ and $\gamma \in J_{\alpha} \setminus J$ implies $\beta <_{J_{\alpha}} \gamma$. 

Now we introduce a new collection of formal symbols $S = \{f_{\alpha,\beta} : \alpha \in I,\beta \in J_{\alpha}\}$ with $f_{\alpha,\beta} = f_{\alpha',\beta}$ for all $\alpha,\alpha' \in I$ and $\beta \in J$. We order these symbols by stipulating that
$$
f_{\alpha,\beta} < f_{\alpha',\beta'} \iff \left\{
\begin{matrix}
    \beta,\beta' \in J \text{ and } \beta <_{J} \beta' \\
    \beta \in J \text{ and }\beta' \not\in J \\
    \alpha = \alpha' \text{ and } \beta <_{J_{\alpha}} \beta' \\
    \beta,\beta' \not\in J \text{ and } \alpha < \alpha'
\end{matrix}
\right.
$$
In other words, the $f_{\alpha,\beta} = f_{\alpha',\beta}$ with $\beta \in J$ are at the bottom of the order and are ordered to match the order on $J$, then the elements $(f_{\alpha,\beta})_{\alpha \in I, \beta \in J_{\alpha} \setminus J}$ are ordered lexicographically. 

Let $\overline{L}$ the free Lie algebra generated by $S$ (see \cite{bourbakiLiechapter23} or \cite{Khu98}).  We form an ideal $Q$ generated by elements of the form $q_{\alpha,\beta,\gamma}$ for $\alpha \in I$ and $\beta <_{J_{\alpha}} \gamma$ in $J_{\alpha}$ defined to be 
$$
q_{\alpha,\beta,\gamma} := [f_{\alpha,\beta},f_{\alpha,\gamma}] - \sum_{\delta \in J_{\alpha}} \lambda_{\delta}f_{\alpha,\delta}
$$
(for scalars $\lambda_{\delta} \in \mathbb{F}$), where the equation 
$$
[e_{\alpha,\beta},e_{\alpha,\gamma}] = \sum_{\delta \in J_{\alpha}} \lambda_{\delta}e_{\alpha,\delta}
$$
is true in the Lie algebra $L_{\alpha}$. 

The amalgamated free product of the Lie algebras $(L_{\alpha})_{\alpha \in I}$ over $L_{*}$ will be defined to be $\overline{L}/Q$. When we are forming the free amalgam over $L_{*}$ of two Lie algebras $A$ and $B$, it will be denoted $A *_{L_{*}} B$.  

It is reasonably clear that, if there is any amalgamated free product of $(L_{\alpha})_{\alpha \in I}$ over $L_{*}$ at all, it should be $\overline{L}/Q$, but the key feature of Shirshov's analysis is that he provides an explicit basis for $\overline{L}/Q$ which renders the analysis of this Lie algebra tractable. 

The free Lie algebra $\overline{L}$ has a basis consisting of the canonical bracketings of regular words in $S$ \cite[Theorem 1]{shirshov1962hypothesis}. The following definition isolates the regular words which form a basis for the quotient $\overline{L}/Q$:

\begin{defn}
    A basis word $v$ of $\overline{L}$ (i.e. a canonical bracketing of a regular associative word in $S$) is called \emph{special} if the corresponding regular associative word contains no subword of the form $f_{\alpha,\beta}f_{\alpha,\beta'}$ with $\beta >_{J_{\alpha}} \beta'$. 
\end{defn}

\begin{fact} \label{fact: shirshov basis} \cite[Theorem 1]{shirshov1962hypothesis}
    The Lie algebra $\overline{L}/Q$ has a basis formed by the (images of the) special regular words.
\end{fact}

Note that, as each word in $S$ of length $1$ is automatically a regular special word, it follows that each Lie algebra $L_{\alpha}$ embeds into $\overline{L}/Q$ and the Lie algebras $L_{\alpha}$ generate $\overline{L}_{\alpha}/Q$ (identifying each $L_{\alpha}$ with the span of $\{f_{\alpha,\beta} : \beta \in J_{\alpha}\}$). Similarly, $\overline{L}/Q$ has the universal property expected of an amalgamated free product: given Lie algebra homomorphisms $\varphi_{\alpha}: L_{\alpha} \to M$ for all $\alpha \in I$ such that $\varphi_{\alpha}|_{L_{*}} = \varphi_{\alpha'}|_{L_{*}}$ for all $\alpha,\alpha' \in I$, there is a unique $\tilde{\varphi}: \overline{L}/Q \to M$ extending each $\varphi_{\alpha}$. 

\subsection{The free amalgam of Lazard Lie Algebras}

Suppose $B$ is a Lazard Lie algebra over a field $\mathbb{F}$, with Lazard structure given by $(P_i)_{1\leq i\leq c+1}$. If $b \in B \setminus \{0\}$, we define the \emph{level} of $b$, denoted $\mathrm{lev}(b)$, to be the maximal $i$ such that $b \in P_{i}$, see \cite{DINCI}. 

Our goal is to define the amalgamated free product in the category of $c$-nilpotent LLAs. To do this, we will need to choose a specific basis:

\begin{definition}\label{def:Malcev}
    A tuple $b =(b_1,\ldots,b_n)\in B$ is called a \textit{Malcev} tuple over an LLA $A\seq B$ (or simply \textit{Malcev} over $A$) if $b$ is linearly independent over $A$ and for all $i = 1,\ldots, n$ we have 
    \[\Span_\F(AP_i(\vect{Ab})) = \Span_\F(A P_i(b)).\]
    Here we write $P_{i}(b)$ for the subtuple of $b$ contained in $P_{i}$. 
 If $B = \vect{Ab}$ we call $b$ a \textit{Malcev basis} of $B$ over $A$. When $A = 0$, we just refer to $b$ as a \emph{Malcev basis}. 
\end{definition}

\begin{lem} \label{lem:malcev}
Let $A$ and $B$ be $c$-nilpotent LLAs. 
\begin{enumerate}
    \item If $A \subseteq B$, then a Malcev basis for $A$ can be extended to a Malcev basis for $B$ over $A$. In particular, a Malcev basis always exists for any $c$-nilpotent LLA.
    \item If $a$ is a Malcev basis for $A$ and $b$ is a Malcev basis for $B$ over $A$, then $ab$ is a Malcev basis for $B$. 
    \item If $a_{1},\ldots, a_{n}$ is a Malcev basis for $A$, then for all non-zero $a \in A$, if 
    $$
    a = \sum_{i=1}^{n} \lambda_{i} a_{i},
    $$
    then $\mathrm{lev}(a) = \min \{\mathrm{lev}(a_{k}) : \lambda_{k} \neq 0\}$. 
\end{enumerate}
\end{lem}

\begin{proof}
(1) Starting with a basis $a$ for $A$, we inductively extend $a$ first to a basis $ab_{c}$ for $\mathrm{span}_{\mathbb{F}}(A P_{c}(B))$, then extend $ab_{c}$ to a basis $ab_{c}b_{c-1}$ for $\mathrm{span}_{\mathbb{F}}(AP_{c-1}(B))$, and so on, at each step extending the basis $ab_{c}\overline{a}_{c-1}\ldots b_{i+1}$ for $\mathrm{span}_{\mathbb{F}}(AP_{i+1}(B))$ to a basis for the larger space $\mathrm{span}_{\mathbb{F}}(AP_{i}(B))$. Since we have $P_{c}(B) \subseteq P_{c-1}(B) \subseteq \ldots \subseteq P_{1}(B) = B$, we eventually obtain a Malcev basis for $B$ over $A$. The existence of a Malcev  basis for $B$ follows by taking $A = 0$ in this argument. 

(2) This is \cite[Lemma 4.28 (1)]{DINCI}.

(3) By the Malcev property, a non-zero $a \in A$ is in $P_{i}(A)$ if and only if $a$ is in the span of $P_{i}(a_{1},\ldots, a_{n})$. Hence, the level of $a$ is $i$ if and only if $a$ is in the span of $P_{i}(a_{1},\ldots, a_{n})$, but not in the span of $P_{i+1}(a_{1},\ldots, a_{n})$. This entails that the level of $a$ is $i$ if only if $\min \{\mathrm{lev}(a_{k}) : \lambda_{k} \neq 0\} = i$.
\end{proof}

Now we construct the free amalgam of $c$-nilpotent Lazard Lie algebras as follows. Suppose, keeping the notation of the previous subsection, $(L_{\alpha})_{\alpha \in I}$ is a family of $c$-nilpotent LLAs over $\mathbb{F}$ with a common sub-LLA $L_{*}$. We choose a Malcev basis $\{e_{\beta}: \beta \in J\}$ for $L_{*}$ and then we write $e_{\alpha,\beta} = e_{\beta}$ for all $\alpha \in I$ and $\beta \in J$. Next, for each $\alpha \in I$, we extend $\{e_{\alpha,\beta} : \beta \in J\}$ to a Malcev basis $\{e_{\alpha,\beta} : \alpha \in I, \beta \in J_{\alpha}\}$ for $L_{\alpha}$. This is possible by Lemma \ref{lem:malcev}(1) and (2). Then, as in the previous section, we form the formal set of generators $S = \{f_{\alpha,\beta}: \alpha \in I, \beta \in J_{\alpha}\}$ with $f_{\alpha,\beta} = f_{\alpha',\beta}$ for all $\alpha,\alpha' \in I$ and $\beta \in J$ and add the ordering as defined in the previous subsection. We obtain a function $\mathrm{lev} : S \to \mathbb{N}^{+}$ defined by
$$
\mathrm{lev}(f_{\alpha,\beta}) = \mathrm{lev}(e_{\alpha,\beta}),
$$
where the $\mathrm{lev}$ on the right-hand side refers to the level of $e_{\alpha,\beta}$ in the Lazard Lie algebra $L_{\alpha}$. Then for each associative word $v = a_{1}\ldots a_{n} \in S^{*}$, we define $\mathrm{lev}(v) = \sum_{i = 1}^{n} \mathrm{lev}(a_{i})$. If $[v]$ is the canonical bracketing of a regular associative word $v$, we define $\mathrm{lev}([v]) = \mathrm{lev}(v)$ (i.e. the level of a regular non-associative word is the same as the level of the corresponding associative word). Note that for a regular associative word (and its canonical bracketing $[w]$), the level assigned to $w$ is the sum of the levels of the content of $w$ (taking into account their multiplicity). 

If $b \in \overline{L}/Q$ is an element of the basis defined by Shirshov, we define $\mathrm{lev}(b)$ to be $\mathrm{lev}([v])$ for the unique special regular word (viewed as an element of a basis for $\overline{L}$) whose image in $\overline{L}/Q$ is $b$.  Finally, for an arbitrary element $\ell \in \overline{L}/Q$, we may express $\ell$ uniquely as a linear combination of (images of) regular special words 
$$
\ell = \sum_{i} \lambda_{i} b_{i},
$$
where $\lambda_{i} \in \mathbb{F}^{\times}$ and $b_{i}$ is the image of a regular special word in the basis for $\overline{L}$. Then we define $\mathrm{lev}(\ell) = \min_{i} \mathrm{lev}(b_{i})$. 

Now we define a series $(P_{i})_{i = 1}^{\infty}$ on $\overline{L}/Q$ by setting $P_{i} = \{ \ell \in \overline{L}/Q : \mathrm{lev}(\ell) \geq i\}$. By the proof of Lemma \ref{lem:malcev} (3), we clearly have 
\[P_i = \Span( \{b |b\textit{ is a regular special word with $\lev(b)\geq i$}\}).\] This turns out to be an infinite length Lazard series.

\begin{lem}
For all $n,m \in \mathbb{N}^{+}$, we have $[P_{n},P_{m}] \leq P_{n+m}$ in $\overline{L}/Q$.
\end{lem}

\begin{proof}
 It suffices to show that if $w \in P_{n}$ and $w' \in P_{m}$ are regular special words, then $\mathrm{lev}([w,w']) \geq n+m$. The free Lie algebra $\overline{L}$ has a basis consisting of the regular words. Shirshov shows, moreover, that a non-associative word (viewed as an element in $\overline{L}$) can be expressed as a linear combination of regular words with the same content (\cite[Lemma 1]{shirshov1958free}), so we may write 
 $$
 [w,w'] = \sum_{i = 1}^{k} \lambda_{i}v_{i},
 $$
 where each $v_{i}$ is a regular word with the same content as $[w,w']$.  Thus, in order to prove the lemma, it suffices to show the following.

 \textbf{Claim}: Let $v$ be a nonassociative word and let $N$ be the sum of the levels of the elements of the content of $v$ added with multiplicity, then in $\overline{L}/Q$, the word $v$ may be expressed as 
 $$
 v = \sum_{i = 1}^{\ell} \mu_{i} u_{i}
 $$
 where each $u_{i}$ is a regular special word with $\mathrm{lev}(u_{i}) \geq N$. 

\emph{Proof of Claim}: We will argue by induction, assuming it is true for all words (not necessarily regular) of smaller degree and for words of the same degree, but which are lexicographically smaller. Note that the claim is obviously true for words of length $1$, which are automatically regular and special. 

As any nonassociative word can be written as a linear combination of regular words with the same content, we may assume $v$ is regular. In \cite[proof of Theorem 1]{shirshov1962hypothesis}, Shirshov observes that a regular word $v$ is not special if one of the three cases holds:
\begin{enumerate}
    \item $v$ has a subword of the form $[f_{\alpha,\beta}, f_{\alpha,\gamma}]$ for $\beta > \gamma$. 
    \item $v$ has a subword of the form $[f_{\alpha,\beta},[f_{\alpha,\gamma},w]]$ where $\beta > \gamma$ and $w$ is a regular word. 
    \item $v$ has a subword of the form $[f_{\alpha,\beta},[u,u']]$ where the regular associative word corresponding to $u$ starts with $f_{\alpha,\gamma}$, where $\beta > \gamma$. 
\end{enumerate}
In Case (1), working modulo $Q$, we may rewrite 
$$
[f_{\alpha,\beta},f_{\alpha,\gamma}] = \sum_{\tau} \lambda_{\tau} f_{\alpha,\tau},
$$
as a linear combination of words obtained by replacing the occurrence of $[f_{\alpha,\beta},f_{\alpha,\gamma}]$ with $f_{\alpha,\tau}$, each of which has lower degree. As we chose the basis to be Malcev, we know that $\mathrm{lev}(f_{\alpha,\tau})\geq \mathrm{lev}(f_{\alpha,\beta})+\mathrm{lev}(f_{\alpha,\gamma})$, whence the induction hypothesis yields the claim.

In Case (2), we apply the Jacobi identity:
$$
[f_{\alpha,\beta},[f_{\alpha,\gamma},w]] = [[f_{\alpha,\beta},f_{\alpha,\gamma}],w] + [f_{\alpha,\gamma},[f_{\alpha,\beta},w]]. 
$$
Note that both summands have the same content as the left hand side. As in (1), we may replace the term $[[f_{\alpha,\beta},f_{\alpha,\gamma}],w]$ with a linear combination of regular terms with smaller degree.  The term $[f_{\alpha,\gamma},[f_{\alpha,\beta},w]]$ is lexicographically smaller than $[f_{\alpha,\beta},[f_{\alpha,\gamma},w]]$ and thus, replacing the occurrence of $[f_{\alpha,\beta},[f_{\alpha,\gamma},w]]$ in $v$ with $[f_{\alpha,\gamma},[f_{\alpha,\beta},w]]$ results in a word with the same content that is lexicographically smaller. Thus Case (2) follows by induction. 

In Case (3), we again use the Jacobi identity, to write 
$$
[f_{\alpha,\beta},[u,u']] = [[f_{\alpha,\beta},u],u'] + [u,[f_{\alpha,\beta},u']]. 
$$
Since $u$ is regular starting with $f_{\alpha,\gamma}$ and $\beta > \gamma$, we know $[f_{\alpha,\beta},u]$ is regular and falls into one of Cases (1), (2), or (3). As the degree of $[f_{\alpha,\beta},u]$ is smaller than $v$, the induction hypothesis entails that we may write $[f_{\alpha,\beta},u]$ as a linear combination of regular special words which are either of smaller degree or lexicographically smaller. Expanding by linearity, then, expresses $[[f_{\alpha,\beta},u],u']$ as a linear combination of terms either of smaller degree or which are lexicographically smaller. The term $[u,[f_{\alpha,\beta},u']]$ starts in $f_{\alpha,\gamma}$ and thus the resulting term is also lexicographically smaller than $v$. Thus Case (3) additionally follows by induction, completing the proof.
\end{proof}

It follows from the above lemma that, in particular, each $P_{i}$ is an ideal of $\overline{L}/Q$. Now we form a new Lie algebra $L' = (\overline{L}/Q)/P_{c+1}$ and view it as an $c$-nilpotent LLA via interpreting the predicates as the images of the $P_{i}$ for $1 \leq i \leq c+1$. 

Now we specialize to the situation that $I = \{1,2\}$ and we label $C = L_{*}$, $A = L_{1}$ and $B = L_{2}$.  In this case, we want to show that $L' \cong A \otimes_{C} B$ in the sense of Definition \ref{def:freeamalgambaudisch}.

\begin{lem}
    We have an isomorphism of LLAs: $L' \cong A \otimes_{C} B$. 
\end{lem}

\begin{proof}
    First, we observe that $A$, $B$, and $C$ embed into $L'$.  To see this, note that $P_{c+1}$ is spanned by the regular special words $v$ with $\mathrm{lev}(v) \geq c+1$. Each basis element $f_{1,\beta}$ of $A$ is a regular special word of level $\leq c$ and hence the elements $\{f_{1,\beta} : \beta \in J_{1}\}$ remain linearly independent in the quotient $(\overline{L}/Q)/P_{c+1}$. It follows that the map sending $e_{1,\beta}$ to the image of $f_{1,\beta}$ for each $\beta \in J_{1}$ is an embedding of Lie algebras from $A$ into $L'$ and it is clear from the construction and Lemma \ref{lem:malcev}(3) that it respects the LLA structure on $A$. A symmetric argument implies that the map sending $e_{2,\beta}$ to the image of $f_{2,\beta}$ in $L'$ is also an embedding of LLAs. Both of these embeddings restrict to give the same embedding of $C$ into $L'$. Moreover, the linear independence of regular special words entails that the image of $A$ and $B$ in $L'$ intersect exactly in the image of $C$. 

    The embeddings of $A$ and $B$ into $L'$ (which agree on $C$) induce a unique homomorphism of LLAs $\varphi: A \otimes_{C} B \to L'$. We also have a map of Lie algebras $\overline{L}/Q \to A \otimes_{C} B$ by the universal property of the amalgamated free product of Lie algebras. By an easy induction, the image of $P_{i}$ in $A \otimes_{C} B$ is contained in $P_{i}(A \otimes_{C} B)$ so this map induces an LLA homomorphism $\psi: L' \to A \otimes_{C}B$. The compositions $\varphi \circ \psi$ and $\psi \circ \varphi$ induce the identity on $A$ and $B$ and thus are the identity maps on $L'$ and $A \otimes_{C} B$, respectively. This shows $L' \cong A \otimes_{C} B$. 
\end{proof}

\section{Topological genericity}

Ivanov observed that if $G$ is a finite $c$-nilpotent group of exponent $p$ (for $p > c$) with $\gamma_{i}(G) = Z_{c+1-i}(G)$ for $i = 1, \ldots, c$, then any embedding $\varphi: G \to H$ to a $c$-nilpotent group $H$ will necessarily respect any Lazard structure placed on $G$, since the Lazard series for $G$ is unique, see Lemma \ref{lm:lazardUL}. Thus, if we can show that if $G$ is a finite $c$-nilpotent group of exponent $p$ (with $p > c$), there is a finite $c$-nilpotent group $G_{*}$ with $G \leq G_{*}$ and $\gamma_{i}(G_{*}) = Z_{c+1-i}(G_{*})$ for all $i = 1,\ldots, c$, then we would obtain the cofinal amalgamation property (so in particular the weak amalgamation property) for the class of $c$-nilpotent groups of exponent $p$ in the group language. This entails, by Kabluchko and Tent \cite{kabluchko2022weakfraisselimits} that there is a generic isomorphism type (in the sense of an isomorphism type being comeager in the space of all $c$-nilpotent groups of exponent $p$ with underlying set $\omega$). 

We will solve the corresponding problem for Lie algebras:
\begin{theorem}\label{thm:mainliealgebra}
    Let $A$ be a nonzero finite-dimensional $c$-nilpotent Lie algebra over $\mathbb{F}$. View $A$ as an LLA by interpreting $P_{i}(A) = \gamma_{i}(A)$. Let $B = \langle x \rangle \otimes_{0} A$ where $\langle x \rangle$ is a $1$-dimensional abelian LLA generated by $x$ and in which $x$ is at level $1$. Then $B$ is finite dimensional and $\gamma_{i}(B) = Z_{c+1-i}(B)$ for all $i = 1, \ldots, c$. 
\end{theorem}

\begin{proof}
    Let $(a_{\alpha})_{\alpha \in J}$ be a Malcev basis for $A$. Then let $S = \{x\} \cup \{a_{\alpha} : \alpha \in J\}$ and well-order $S$ so that $x$ is the greatest element and the least element is $a_{\alpha_{*}} \in A \setminus \gamma_{2}(A)$ (so has level $1$). Then we know that $B$ is spanned by the regular special words from the alphabet $S$ with level $\leq c$. 

    We always have $\gamma_{i}(B) \subseteq P_{i}(B) \subseteq Z_{c+1-i}(B)$, so it is enough to prove $Z_{c+1-i}(B) \subseteq P_{i}(B)$ and $P_{i}(B) \subseteq \gamma_{i}(B)$ for $i = 1, \ldots, c$. 

    First, we show $Z_{c+1-i}(B) \subseteq P_{i}(B)$ for $i = 1, \ldots, c$. This is trivial for $i = 1$ so fix some $i \in \{2, \ldots, c\}$ and assume $\ell \not\in P_{i}(B)$. Write 
    $$
    \ell = \sum_{k = 1}^{n} c_{k} w_{k}
    $$
    where $c_{k} \in \mathbb{F}^{\times}$ and $w_{k}$ is a regular special word.  By the definition of level, we have $\min_{k} \mathrm{lev}(w_{k}) = j < i$. We have to consider two cases.  First, suppose the only regular special word $w_{k}$ with $\mathrm{lev}(w_{k}) < i$ is $w_{k_{*}} = x$. Then $j = 1$. Up to a reordering of terms, we may assume, then, that 
    $$
    \ell = c_{1}x + \sum_{k=2}^{n} c_{k} w_{k}
    $$
    where each $w_{k}$ is a regular special word with $\mathrm{lev}(w_{k}) > 1$. Then we bracket with $a_{\alpha_{*}}$ on the right $c-1$ times with $\ell$ to obtain the following:
    \begin{eqnarray*}
    [[[\ldots [\ell,a_{\alpha_{*}}],a_{\alpha_{*}}],\ldots],a_{\alpha_{*}}] &=& c_{1}[[[\ldots [x,a_{\alpha_{*}}],a_{\alpha_{*}}],\ldots],a_{\alpha_{*}}] + \sum_{k=2}^{n} c_{k}[[[\ldots [w_{k},a_{\alpha_{*}}],a_{\alpha_{*}}],\ldots],a_{\alpha_{*}}] \\
    &=& c_{1}[[[\ldots [x,a_{\alpha_{*}}],a_{\alpha_{*}}],\ldots],a_{\alpha_{*}}],
    \end{eqnarray*}
    where the first equality holds by linearity and the second follows because each term $[[[\ldots [w_{k},a_{\alpha_{*}}],a_{\alpha_{*}}],\ldots],a_{\alpha_{*}}]$ has level $c-1 + \mathrm{lev}(w_{k}) > c$ and hence must equal $0$. It is clear that $[[[\ldots [x,a_{\alpha_{*}}],a_{\alpha_{*}}],\ldots],a_{\alpha_{*}}]$ is the canonical bracketing of the associative regular special word $xa_{\alpha_{*}}a_{\alpha_{*}}\ldots a_{\alpha_{*}}$ with $a_{\alpha_{*}}$ occurring $c-1$ times. Its level is $\mathrm{lev}(x) + (c-1)\mathrm{lev}(a_{\alpha_{*}}) = c$. Thus this term is not zero and we obtain that $\ell \not\in Z_{c-1}(B)$, hence $\ell \not\in Z_{c+1-i}(B)$.

    The next case we consider is just \emph{not} the previous case, so there is at least one regular special word $w_{k} \neq x$ with $\mathrm{lev}(w_{k}) = j = \min_{k} \mathrm{lev}(w_{k})< i$. Let $X \subseteq \{1, \ldots, n\}$ denote the set of $k$ such that $w_{k} \neq x$ and $\mathrm{lev}(w_{k}) < i$. Then we take right-normed brackets with $c+1-i$ many $x$ to obtain 
    \begin{eqnarray*}
        [x,[x,\ldots[x,\ell]\ldots]] &=& \sum_{k = 1}^{n} c_{k}[x,[x,\ldots[x,w_{k}]\ldots]] \\
        &=& \sum_{k \in X} c_{k}[x,[x,\ldots[x,w_{k}]\ldots]]
    \end{eqnarray*}
    Now for each $k \in X$, the term $[x,[x,\ldots[x,w_{k}]\ldots]]$ is a regular special word (since we chose $x$ to be the greatest element of $S$) of level $\mathrm{lev}(w_{k}) + (c+1-i) \leq c$. Thus these terms are linearly independent in $B$ and we obtain $[x,[x,\ldots[x,\ell]\ldots]] \neq 0$. This shows $\ell \not\in Z_{c+1-i}(B)$. These two cases together show that $Z_{c+1-i}(B) \subseteq P_{i}(B)$. 

    Next we show for each $i = 1, \ldots, c$ that $P_{i}(B) \subseteq \gamma_{i}(B)$. It suffices to show that if $w \in P_{i}(B)$ is a regular special word, then $w \in \gamma_{i}(B)$. We prove this by induction on the length of $w$. For $w$ of length $1$, this is immediate since either $w = x$ so $w \in P_{1}(B) \setminus P_{2}(B)$ and there is nothing to show, or $w = a_{\alpha}$ for some $\alpha \in J$ in which case $w \in P_{i}(B)$ if and only if $w \in \gamma_{i}(A) \subseteq \gamma_{i}(B)$ by construction. Now suppose the desired containment was shown for all regular special words of length less than that of $w$ and write the corresponding word $w = uv$ with $v$ the maximal length regular proper suffix of $w$. Then both $u$ and $v$ are regular special words and the canonical bracketing of $w$ is $[w] = [[u],[v]]$. Then, by construction, $w \in P_{i}(B)$ if and only if $u \in P_{j}(B)$ and $v \in P_{k}(B)$ for some $j,k$ with $j + k = i$. This latter condition implies $u \in \gamma_{j}(B)$ and $v \in \gamma_{k}(B)$ by induction, and this implies $w = [u,v] \in \gamma_{j+k}(B) = \gamma_{i}(B)$ since the lower central series is a Lazard series. This completes the proof. 
\end{proof}

\subsection{Translating to groups}

Recall that $\mathbb{K}_{c,p}$ denotes the class of groups of exponent $p$ and nilpotency class at most $c$ (viewed as structures in the language of groups) and $\mathbb{K}^{P}_{c,p}$ denotes the same class of groups, but in a language extending the language of groups with $c+1$ predicates $(P_{i})_{i = 1}^{c+1}$, which are interpreted to name a Lazard series. We analyzed $\mathbb{K}^{P}_{c,p}$ in \cite{DINCI}, where we showed that this class comes with a notion of free amalgamation, which allowed for model-theoretic applications. In this section, we will leverage the results of the previous section to relate the generic $c$-nilpotent group of exponent $p$ to the Fra\"iss\'e limit of $\mathbb{K}_{c,p}$.  

In parallel, we will obtain similar applications for torsion-free $c$-nilpotent groups. Let $\mathbb{K}_{c,\mathrm{tf}}$ denote the class of finitely generated torsion-free groups of nilpotency class $\leq c$. We view each $G \in \mathbb{K}_{c,\mathrm{tf}}$ as a structure in the language of groups. Torsion-free nilpotent groups are often approached by studying their Malcev completion (see Definition Fact \ref{fact: malcev completion}(3) below), which is an analogue of the divisible hull of a torsion-free abelian group. A group $G$ is called $\mathbb{Q}$-powered if it is torsion-free and, for all $g \in G$ and natural numbers $n \geq 2$, $x^{n} = g$ has a unique solution in $G$. Let $\mathbb{K}_{\mathbb{Q},c}$ denote the class of finitely generated $\mathbb{Q}$-powered groups of nilpotency class $\leq c$. We view each $G \in \mathbb{K}_{\mathbb{Q},c}$ as a structure in the language of groups, together with a unary function symbol $\sqrt[n]{-}$ interpreted as the unique function maps each $g \in G$ to the unique $h$ with $h^{n} = g$, for all $n \geq 2$. We refer to this language as the \emph{language of }$\mathbb{Q}$\emph{-powered groups}. Let $\mathbb{K}_{\mathbb{Q},c}^{P}$ denote the class of finitely generated $\mathbb{Q}$-powered groups of nilpotency class $\leq c$ in a language expanded to include predicates for a Lazard series. We write $\mathbb{G}_{c,\mathbb{Q}}$ for the Fra\"iss\'e limit of $\mathbb{K}^{P}_{c,\mathbb{Q}}$.

We will analyze nilpotent groups by translating questions about them into questions about Lie algebras. For the case of $c$-nilpotent groups of exponent $p$, for a prime $p > c$, the machinery underlying this translation is \emph{Lazard correspondence}, which associates to each nil-$c$ group of exponent $p$ a nil-$c$ Lie algebra over the field $\mathbb{F}_{p}$. In \cite{DINCI}, it was pointed out that this correspondence gives the uniform bi-definability of nil-$c$ groups of exponent $p$ and of nil-$c$ Lie algebras over $\mathbb{F}_{p}$ (in both to the pure languages of groups and Lie algebras and to their respective expansions to languages with predicates for Lazard series).  The focus in \cite{DINCI} on Lie algebras over $\mathbb{F}_{p}$ was natural in the setting of studying $\aleph_{0}$-categorical Fra\"iss\'e limits, but here will will make parallel use of a similar correspondence, the \emph{Malcev correspondence}, 

Suppose $c < p$, for $p$ an odd prime.  If $G$ is a group of nilpotency class $\leq c$ which is either of exponent $p$ or $\mathbb{Q}$-powered, we define $L_{G}$ to be a structure with same underlying set and operations $+_{L_{G}}$ and $[\cdot,\cdot]_{L_{G}}$ defined by 
$$
g +_{L_{G}} h = h_{1}(g,h) =  gh[g,h]^{-\frac{1}{2}}[g,g,h]^{-\frac{1}{12}}[h,g,h]^{\frac{1}{12}}\ldots 
$$
and 
$$
[g,h]_{L_{G}} = h_{2}(g,h) = [g,h][g,g,h]^{\frac{1}{2}}[h,g,h]^{\frac{1}{2}}\ldots 
$$
where the brackets on the right denote group commutators in $G$. We will usually omit the subscripts. Since the nilpotence class of $G$ is at most $c$, it turns out that both $h_{1}$ and $h_{2}$ are finite products of group commutators in $G$ raised to powers in $\mathbb{Z}_{(p)}$, where $\mathbb{Z}_{(p)}$ denotes the set of $q \in \mathbb{Q}$ such that if $q = \frac{l}{m}$ is in reduced form, then $\mathrm{gcd}(m,p) = 1$ (in the case that $G$ has exponent $p$, it is crucial that $p$ does not divide the denominator, but, in the case that $G$ is torsion-free, it only matters that the exponents are in $\mathbb{Q}$). Since the group $G$ is either is of exponent $p$ or $\mathbb{Q}$-powered, it makes sense to raise any element to powers in $\mathbb{Z}_{(p)}$ and these yield well-defined operations on $L_{G}$.   The coefficients of $h_{1}$ and $h_{2}$ are explicitly described in \cite{Cicaloetal}.  

Conversely, given a Lie algebra $L$ over $\mathbb{F}_{p}$ or $\mathbb{Q}$ of nilpotency class at most $c$, one defines $G_{L}$ to be the structure with the same underlying set and with a binary operation $*_{G_{L}}$ defined by 
$$
a *_{G_{L}} b = H(a,b) =  a + b + \frac{1}{2} [a,b] + \frac{1}{12}[a,a,b] - \frac{1}{12} [b,a,b] + \ldots 
$$
This is the Baker-Campbell-Hausdorff formula, where the brackets on the right-hand side are the Lie bracket of $L$. This is usually an infinite sum but, since $L$ is of nilpotency class  at most $c$, the function $H$ can be written as a finite linear combination of Lie monomials with coefficients in $\mathbb{Z}_{(p)}$.  This is a $\mathbb{Q}$-linear combination of monomials but can also be reduced mod $p$ can thus be viewed as an $\mathbb{F}_{p}$-linear combination of Lie monomials.  

The following fact simultaneously summarizes the Malcev and Lazard correspondences.

\begin{fact} \cite[Chapter 10]{Khu98}
Suppose $c < p$ for an odd prime $p$.  To every $\mathbb{Q}$-powered group $G$ (group of exponent $p$) and nilpotence class $\leq c$, the Malcev correspondence (Lazard correspondence) associates a Lie algebra $L_{G}$ over $\mathbb{Q}$ (over $\mathbb{F}_{p}$) with the same underlying set $L_{G} = G$ and with operations $a+b = h_{1}(a,b)$ and $[a,b] = h_{2}(a,b)$, and $ra = a^{r}$ for every $r \in \mathbb{Q}$ ($r \in \mathbb{F}_{p}$).  Conversely, for every Lie algebra $L$ over $\mathbb{Q}$ (over $\mathbb{F}_{p}$) of nilpotency class $\leq c$, there is a corresponding group $G_{L}$ which is $\mathbb{Q}$-powered (of exponent $p$) with the same underlying set and group operation defined by 
$$
a * b = H(a,b)
$$
and $a^{r} = ra$ for $r \in \mathbb{Q}$ ($r \in \mathbb{F}_{p}$). These operations are inverses to each other:  as Lie algebras over $\mathbb{Q}$ (over $\mathbb{F}_{p}$), we have $L_{G_{L}} = L$ and additionally $G_{L_{G}} = G$ as groups. 
\end{fact}

The following summarizes the key facts that we need about the Lazard correspondence.

\begin{fact} \label{Lazard facts} \cite[Chapter 10]{Khu98}
    Suppose $c < p$ for an odd prime $p$. Suppose that $L$ is a Lie algebra over $\mathbb{F}_{p}$ of nilpotence class $\leq c$, that $G$ is a group of nilpotence class $\leq c$ of exponent $p$, and that $L$ and $G$ are in correspondence with one another, i.e. $L = L_{G}$ as Lie algebras and $G = G_{L}$ as groups.
    \begin{enumerate}
        \item For all $a,b \in L$, 
        $$
        [a,b]_{L} = [a,b]_{G} \prod_{j} \chi_{j}^{s_{j}}
        $$
        where $s_{j} \in \mathbb{Z}_{(p)}$ and $\chi_{j}$ are group commutators in $a$ and $b$ of degree $\geq 3$.
        \item For all $a,b \in G$, 
        $$
        [a,b]_{G} = [a,b]_{L} + \sum_{j} u_{j} \chi_{j}
        $$
        where $u_{j} \in \mathbb{F}_{p}$ and the $\chi_{j}$ are Lie monomials in $a$ and $b$ of degree $\geq 3$. 
        \item A subset $K \subseteq G$ is a subgroup of $G$ if and only if $K \subseteq L$ is a Lie subalgebra. 
        \item A subset $I$ is a normal subgroup of $G$ if and only if $I$ is an ideal of $L$.
        \item A function from the underlying set $G = L$ to itself is an endomorphism of the group $G$ if and only if it is an endomorphism of the Lie algebra $L$. In particular, the automorphism groups $\mathrm{Aut}(L)$ and $\mathrm{Aut}(G)$ coincide as permutation groups. 
    \end{enumerate}
\end{fact}

Note that Fact \ref{Lazard facts} implies if $L$ is a Lie algebra over $\mathbb{F}_{p}$ or $\mathbb{Q}$ of nilpotence class $\leq c$, the group $G$ is of nilpotency class $\leq c$ and is of exponent $p$ or is $\mathbb{Q}$-powered, and $L$ and $G$ are in correspondence with one another, then a sequence $(H_{i})_{1 \leq i \leq c+1}$ is a Lazard series for $G$ if and only if it is a Lazard series for $L$.  

\begin{lem}
    The class $\mathbb{K}^{P}_{c,\mathbb{Q}}$ of finitely generated $c$-nilpotent $\mathbb{Q}$-powered groups, with a named Lazard series, is a Fra\"iss\'e class. 
\end{lem}

\begin{proof}
    This follows as in \cite[Corollary 4.39]{DINCI}, using the fact that each $G \in \mathbb{K}_{c,\mathbb{Q}}$ is quantifier-free bi-definable with a Lazard Lie algebra over $\mathbb{Q}$. 
\end{proof}

If $H$ is a subgroup of a group $G$, then we define $\sqrt{H}$ to be $\{g \in G : g^{n} \in H \text{ for some }n\}$. 

\begin{fact} \label{fact: malcev completion}
    Suppose $G$ is a nilpotent group. 
    \begin{enumerate}
        \item If $H \leq G$, then $\sqrt{H}$ is a subgroup. If $A,B \leq G$ and $A \unlhd B$, then $\sqrt{A} \unlhd \sqrt{B}$. \cite[Theorem 9.18]{Khu98}
        \item If $G$ is $\mathbb{Q}$-powered and $H \leq G$ is an abstract subgroup, then $\sqrt{H}$ is the $\mathbb{Q}$-powered subgroup of $G$ generated by $H$; conversely if $H \leq G$ is a $\mathbb{Q}$-powered subgroup generated by $X$, then $H = \sqrt{K}$ for $K$ the abstract subgroup generated by $X$. \cite[Corollary 9.19]{Khu98}
        \item If $G$ is torsion-free, then $G$ can embedded into a $\mathbb{Q}$-powered group $\hat{G}$ of the same nilpotency class such that $\sqrt{G} = \hat{G}$. The group $\hat{G}$ is called the \emph{Malcev completion} of $G$. It is unique up to isomorphism and every isomorphism $G \to G'$ extends to an isomorphism $\hat{G} \to \hat{G'}$. \cite[Theorem 9.20]{Khu98}
    \end{enumerate}
\end{fact}

\begin{prop}
    The classes $\mathbb{K}_{c,p}$, $\mathbb{K}_{c,\mathrm{tf}}$, and $\mathbb{K}_{c,\mathbb{Q}}$ have the cofinal amalgamation property.
\end{prop}

\begin{proof}
    We showed in Theorem \ref{thm:mainliealgebra} that if $L_{0}$ is a finite dimensional $c$-nilpotent Lie algebra over a field $K$, then there is some finite dimensional $L_{1} \supseteq L_{0}$ which is also finite dimensional and $c$-nilpotent and which is additionally a UL-Lie algebra. By Lazard correspondence, if $G_{0} \in \mathbb{K}_{c,p}$, then it follows there is some UL-group $G_{1} \in \mathbb{K}_{c,p}$ with $G_{0} \subseteq G_{1}$. Likewise, by the Malcev correspondence, if $G_{0} \in \mathbb{K}_{c,\mathbb{Q}}$ there is some UL-group $G_{1} \in \mathbb{K}_{c,\mathbb{Q}}$. A UL-group has a unique Lazard series, so given an amalgamation problem with base $G_{1}$ in either case, 
    $$
    
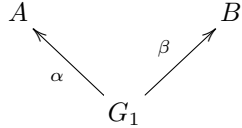
\xymatrix{
A && B \\
& G_{1} \ar[ul]^{\alpha} \ar[ur]^{\beta} &
}
    $$
    then, giving $A$ and $B$ the structure of a Lazard group arbitrarily (by, e.g., naming the lower central series), the maps $\alpha$ and $\beta$ necessarily respect the Lazard predicates and, therefore, there is an amalgam. 

    Now we handle the case of $\mathbb{K}_{c,\mathrm{tf}}$.  Let $A \in \mathbb{K}_{c,\mathrm{tf}}$ and let $\hat{A}$ denote the Malcev completion of $A$. By Fact \ref{fact: malcev completion}(3), $\hat{A} \subseteq \hat{B}$ where $\hat{B}$ is a $c$-nilpotent $\mathbb{Q}$-powered group which is a UL-group that is finitely generated as a $\mathbb{Q}$-powered group. Let $X$ be a finite generating set of $\hat{B}$ (as a $\mathbb{Q}$-powered group) which contains generators for $A$ (as an abstract subgroup). Let $B := \langle X \rangle$ be the abstract subgroup of $\hat{B}$ generated by $X$ (our notation is sensible since $\sqrt{B} = \hat{B}$). 

    We will argue that $B$ is an amalgamation base for $\mathbb{K}_{c,\mathrm{tf}}$. Let $C,D \in \mathbb{K}_{c,\mathrm{tf}}$ and suppose we are given embeddings $\alpha : B \to C$ and $\beta: B \to D$. By Fact \ref{fact: malcev completion}(3), these embeddings lift to embeddings $\hat{\alpha} : \hat{B} \to \hat{C}$ and $\hat{\beta}: \hat{B} \to \hat{D}$. Let $\hat{E}$ be an amalgam for $\hat{C}$ and $\hat{D}$ over $\hat{B}$. Let $E$ be an abstract subgroup of $\hat{E}$ containing a generating set for $\hat{E}$ (as a $\mathbb{Q}$-powered group), as well as the images of $C$ and $D$ in $\hat{E}$. Then $E$ is torsion-free, $c$-nilpotent, and amalgamates $C$ and $D$ over $B$.  
\end{proof}

\subsection{Characterizing amalgamation bases}

In this subsection, we characterize amalgamation bases for the classes $\mathbb{K}_{c,p}$, $\mathbb{K}_{c,\mathbb{Q}}$, and $\mathbb{K}_{c,\mathrm{tf}}$. 

\begin{fact} \label{fact: commutator of center}
    In any group $G$, for all $i,j$, $[\gamma_{i}(G),Z_{j}(G)] \leq Z_{j-i}(G)$ (where, by convention $Z_{k}(G) =1$ for all $k \leq 0$). 
\end{fact}

\begin{proof}
    The proof is an application of the `Three Subgroups Lemma' \cite[Fact 0.3]{dixon2003analytic}  which states that, for all $A,B,C \leq G$,
    $$
    [A,B,C] \leq [B,C,A][C,A,B]. 
    $$
    The proof argues simultaneously for all $j$, by induction on $i$. Since $\gamma_{1}(G) = G$, the inclusion $[\gamma_{1}(G),Z_{j}(G)] \leq Z_{j-1}(G)$ follows from the definition of $Z_{j}(G)$. Assuming the statement has been established for $i$, we check 
    \begin{eqnarray*}
        [\gamma_{i+1}(G),Z_{j}(G)] = [G,\gamma_{i}(G),Z_{j}(G)] &\leq& [\gamma_{i}(G),Z_{j}(G),G][Z_{j}(G),G,\gamma_{i}(G)] \\
        &\leq& [Z_{j-i}(G),G][Z_{j-1}(G),\gamma_{i}(G)] \\
        &\leq& Z_{j-i-1}(G) Z_{j-i-1}(G) \\
        &=& Z_{j-i-1}(G),
    \end{eqnarray*}
    establishing the claim.
\end{proof}

\begin{defn}
    Suppose $G$ is a $c$-nilpotent group. Say $G$ is \emph{Lazard canonical} if there is a unique Lazard series for $G$ (which must necessarily be the lower central series).  
\end{defn}

\begin{lem}\label{lm:lazardUL}
    Suppose $G$ is a $c$-nilpotent group with Lazard series $(H_{n})_{n \geq 1}$ of length $c$.  Then, for all $1 \leq i \leq c+1$, 
$$
\gamma_{i}(G) \leq H_{i} \leq Z_{c+1-i}(G).
$$
\end{lem}

\begin{proof}
    By definition, we have $\gamma_{1}(G) = H_{1} = G$.  Assuming we have shown $\gamma_{n}(G) \leq H_{n}$ for some $n$, we have 
    $$
    H_{n+1} \geq [H_{n},H_{1}] \geq [\gamma_{n}(G),G] = \gamma_{n+1}(G)
    $$
    which establishes the first inequality. 

    For the next direction, we argue by induction on $i$ that $H_{c+1-i} \leq Z_{i}(G)$.  As $G$ is $c$-nilpotent, we must have $H_{c+1} = 1 = Z_{0}(G)$, which handles the base case.  Assuming for $i$, we pick any $g \in G$ and $h \in H_{c+1 - (i+1)}$.  Then by the inductive hypothesis, we have 
    $$
    [g,h] \in H_{c+1-i} \leq Z_{i}(G)
    $$
    so $ghZ_{i}(G) = hgZ_{i}(G)$ which shows $h \in Z_{i+1}(G)$. 
\end{proof}

\begin{prop} \label{prop: UL = LC}
    Suppose $G$ is a $c$-nilpotent group. The following are equivalent:
    \begin{enumerate}
        \item $G$ is a UL-group.
        \item $G$ is Lazard canonical.
    \end{enumerate}
\end{prop}

\begin{proof}
    (1)$\implies$(2) is immediate from Lemma \ref{lm:lazardUL}. 

    (2)$\implies$(1) Assume $G$ is not a UL-group. We will show there is a Lazard series for $G$ which is not the lower central series.  Let $k$ be maximal such that $Z_{c+1-k}(G) \setminus \gamma_{k}(G) \neq \emptyset$ and choose some $a \in Z_{c+1-k}(G) \setminus \gamma_{k}(G)$.  Note that, since $a \in Z_{c+1-k}(G)$, for all $i = 1, \ldots, c$, we have, by Fact \ref{fact: commutator of center} and the maximality of $k$, if $g \in \gamma_{i}(G)$, then $[a,g] \in [Z_{c+1}(G),\gamma_{i}(G)] \leq Z_{c+1-k-i}(G) = \gamma_{k+i}(G)$. 

    Define a series $(P_{i})_{i = 1}^{c+1}$ by setting $P_{i}(G) = \langle \gamma_{i}(G),a \rangle$ for all $i \leq k$ and $P_{i}(G) = \gamma_{i}(G)$ for all $i > k$. 

    Now fix $1 \leq i,j \leq c$. We have to show $[P_{i}(G),P_{j}(G)] \leq P_{i+j}(G)$ (where, by convention, $P_{i+j}(G) = 1$ when $i+j > c$).  We have three cases. First, if both $i$ and $j$ are greater than $k$, this is immediate from the fact that the lower central series is a Lazard series. Secondly, if $i,j \leq k$, then it follows from the commutator identity $[xy,z] = [x,z]^{y}[y,z]$ (\cite[Fact 0.2]{dixon2003analytic}) that any normal subgroup containing $[a,\gamma_{j}(G)]$, $[a,\gamma_{i}(G)]$, and $[\gamma_{i}(G),\gamma_{j}(G)]$ will contain $[P_{i}(G),P_{j}(G)]$. By the argument above using Fact \ref{fact: commutator of center}, it follows that $[a,\gamma_{i}(G)]$ and $[a,\gamma_{j}(G)]$ are contained in $\gamma_{k+i}(G)$ and $\gamma_{j+k}(G)$, respectively, which are both contained in $\gamma_{i+j}(G)$. Thus we have shown that $[P_{i}(G),P_{j}(G)] \leq \gamma_{i+j}(G) \leq P_{i+j}(G)$, as desired. The remaining case is when (without loss of generality) $i \leq k$ and $j > k$. But then, as in the previous case, we observe that $\gamma_{i+j}(G)$ is a normal subgroup containing $[a,\gamma_{j}(G)]$ and $[\gamma_{i}(G),\gamma_{j}(G)]$ and therefore containing $[P_{i}(G),P_{j}(G)]$, finishing the proof. 
\end{proof}

\begin{prop}
    Let $G$ be a finite $c$-nilpotent group of exponent $p$ ($c$-nilpotent finitely generated $\mathbb{Q}$-powered group). Then the following are equivalent:
    \begin{enumerate}
    \item $G$ is a UL-group. 
    \item $G$ is Lazard canonical. 
    \item $G$ is an amalgamation base for $\mathbb{K}_{c,p}$ (or $\mathbb{K}_{c,\mathbb{Q}}$).
    \end{enumerate}
\end{prop}

\begin{proof}
    The equivalence of (1) and (2) is Proposition \ref{prop: UL = LC}. (2) implies (3) since if $\alpha: G \to A$ and $\beta: G \to B$ are group embeddings then, after giving $A$ and $B$ arbitrary expansions with named Lazard series, the embeddings $\alpha$ and $\beta$ must be embeddings also in the expanded language. Then the existence of an amalgam follows from the amalgamation property in $\mathbb{G}_{c,p}^{P}$. 

    Finally, we show (3)$\implies$(2).  Suppose $G$ is not Lazard canonical and let $k$ be greatest such that $Z_{c+1-k}(G) \neq \gamma_{k}(G)$ and pick $a \in Z_{c+1-k}(G) \setminus \gamma_{k}(G)$. Let $\ell$ be greatest such that $a \in \gamma_{\ell}(G)$, so, necessarily, $\ell < k$. 
    
    Let $G_{1} \in \mathbb{G}^{P}_{c,p}$ denote the Lazard group with underlying group $G$ with predicates interpreted by the lower central series. Let $G_{2} \in \mathbb{G}^{P}_{c,p}$ denote the Lazard group with underlying group $G$ with predicates interpreted as in the proof of Proposition \ref{prop: UL = LC}, where $a \in P_{k}(G) \setminus P_{k+1}(G)$. Let $A_{i}$ denote the substructure of $G_{i}$ generated by $a$, so in both cases the underlying group is a cyclic group of order $p$ (or a $\Q$-powered cyclic group) but in $A_{1}$ $a$ is at level $\ell$ and, in $A_{2}$, $a$ is at level $k$. 
    
    Let $F$ denote the free $c$-nilpotent group of exponent $p$ (or $\Q$-powered) generated by $x_{1},\ldots, x_{c}$. View $F$ as a Lazard group by interpreting the predicates by the lower central series. Choose elements $b_{1},b_{2} \in F$ by $b_{1} = [x_{1},\ldots, x_{\ell}] \in \gamma_{\ell}(F)$ and $b_{2} = [x_{1},\ldots, x_{k}] \in \gamma_{k}(F)$. Let $B_{i}$ denote the substructure of $F$ generated by $b_{i}$ so, as Lazard groups, we have $A_{i} \cong B_{i}$ for $i = 1,2$. Let $C_{i}$ denote an amalgam of $G_{i}$ and $F$ over $A_{i}$ where the map $A_{i} \to G_{i}$ is the inclusion, and the map $A_{i} \to F$ is the unique map sending $a \mapsto b_{i}$, for $i = 1,2$.  Let $\beta_{i} : G_{i} \to C_{i}$ be the embeddings given by the amalgamation property. 

    Now we forget the Lazard structure and regard $\beta_{1}$ and $\beta_{2}$ as embeddings of groups $\beta_{1}: G \to C_{1}$ and $\beta_{2} : G \to C_{2}$. Suppose towards contradiction that there is an amalgam $D \in \mathbb{G}_{c,p}$ for these embeddings. Let $d$ be the image of $a$ in $D$ and let $F_{1}$ the copy of $F$ in the image of $C_{1}$ in $D$ with generators $y_{1},\ldots, y_{c}$ and let $F_{2}$ be the copy of $F$ in the image of $C_{2}$ in $D$ with generators $z_{1},\ldots, z_{c}$ so that $d = [y_{1},\ldots, y_{\ell}]$ and $d = [z_{1},\ldots, z_{k}]$. 

    Because $d = [y_{1},\ldots, y_{\ell}]$ and, by freeness of $F_{1}$, $[y_{1},\ldots, y_{c}] \neq 1$, we obtain $[d,y_{\ell+1},\ldots, y_{c}] \neq 1$. Because $d = [z_{1},\ldots, z_{k}]$, we obtain 
    $$
    [z_{1},\ldots, z_{k},y_{\ell+1},\ldots, y_{c}] \neq 1.
    $$
    But since $\ell < k$, $[z_{1},\ldots, z_{k}, y_{\ell+1},\ldots, y_{c}] \in \gamma_{c - \ell + k}(D) \leq \gamma_{c+1}(D) = 1$, a contradiction. This proves that there cannot be such amalgam $D$. 
\end{proof}

\subsection{Describing the generic groups}

Let $\mathbb{K}$ be a countable hereditary class of finitely generated $L$-structures satisfying the joint embedding property (JEP), and let $\mathbb{S}$ denote the set of $L$-structures $M$ whose underlying set is $\omega$ such that $\mathrm{Age}(M) \subseteq \mathbb{K}$ and $M$ is not finitely generated. We will always assume $\mathbb{K}$ is \emph{unbounded}, meaning that for each $A \in \mathbb{K}$, there exists some $B \in \mathbb{K}$ with $A \subsetneq B$. For each finitely generated $L$-structure $B$ with $\mathrm{dom}(B) \subseteq \omega$ and $B$ isomorphic to a structure in $\mathbb{K}$, we define a subset $\mathcal{O}_{B} \subseteq \mathbb{S}$ to be the set of $M \in \mathbb{S}$ such that $B$ is a substructure of $M$ (i.e. the restriction of $M$ to $\mathrm{dom}(B)$ is $B$). Taking the subsets of the form $\mathcal{O}_{B}$ to be a basis of open sets turns $\mathbb{S}$ into a topological space. We say $M \in \mathbb{S}$ is \emph{generic} if the set $\{N \in \mathbb{S} : N \cong M\}$ is comeager. If $\mathbb{K}$ is the class of finitely generated $X$s, we refer to such an $M$ as \emph{the generic} $X$. 

We say a countable structure $M$ is a \emph{weak Fra\"iss\'e limit }of the class $\mathbb{K}$ if $\mathrm{Age}(M) = \mathbb{K}$ and $M$ is \emph{weakly-}$\mathbb{K}$\emph{-homogeneous}, which means that for any finite generated $A \subseteq M$, there is a finitely generated $B$ with $A \subseteq B \subseteq M$ such that if $C \in \mathbb{K}$ and $f : B \to C$ is an embedding, there is some embedding $h: C \to M$ such that, naming the inclusions $i : A \to B$  and $j : B \to M$, we have $j \circ i = h \circ f \circ i$. If a weak Fra\"iss\'e limit exists for $\mathbb{K}$, then it is unique up to isomorphism \cite[Theorem 2.5]{kabluchko2022weakfraisselimits}. The existence of a weak Fra\"iss\'e limit is equivalent to $\mathbb{K}$ satisfying the joint embedding property (JEP) and the weak amalgamation property (WAP, see Definition \ref{def:APs}(3)).  See \cite[Chapter 4]{kruckmanthesis2016} for an excellent exposition of the general theory. 

Finally, if $P$ is a property of $L$-structures, we define the game $G_{\mathbb{K}}(P)$ to be the game in which Players I and II alternate playing an increasing sequence of structures in $\mathbb{K}$.
$$
\begin{array}{c|ccccc}
\mathrm{I} 
  & A_{0} 
  & 
  & A_{2}
  & 
  & \cdots \\ \hline
\mathrm{II} 
  & 
  & A_{1}
  & 
  & A_{3}
  & \cdots
\end{array}
$$
Player II wins the game if $M := \bigcup_{i} A_{i}$ has the property $P$. A useful fact is that if $P_{i}$ is a property for each $i < \omega$, then if Player II has a winning strategy in $G_{\mathbb{K}}(P_{i})$ for each $i$, then, by `interleaving' strategies, Player II also has a winning strategy for $G_{\mathbb{K}}\left( \bigwedge_{i} P_{i} \right)$ \cite[Remark 4.1.15]{kruckmanthesis2016}. 

\begin{fact}
    Let $\mathbb{K}$ be a countable unbounded hereditary class of finitely generated $L$-structures and let $\mathbb{S}$ be the set of $L$-structures $M$ with underlying set $\omega$ such that $\mathrm{Age}(M) \subseteq \mathbb{K}$ and $M$ is not finitely generated. The following are equivalent for a structure $M \in \mathbb{S}$:
    \begin{enumerate}
        \item $M$ is generic. 
        \item $M$ is a weak Fra\"iss\'e limit of $\mathbb{K}$. \cite[Theorem 2.5]{kabluchko2022weakfraisselimits}
        \item If $\mathrm{Iso}_{M}$ denotes the property of being isomorphic to $M$, then Player II has a winning strategy in the game $G_{\mathbb{K}}(\mathrm{Iso}_{M})$. \cite[Corollary 7.1]{krawczyk2021games}
    \end{enumerate}
    Moreover, the existence of an $M \in \mathbb{S}$ satisfying the equivalent conditions (1)-(3) is equivalent to $\mathbb{K}$ satisfying JEP and WAP. 
\end{fact}

Recall that $\mathbb{G}_{c,p}$ is the Fra\"iss\'e limit of the class $\mathbb{K}^{P}_{c,p}$. Write $\mathbb{G}_{c,\mathbb{Q}}^{P}$ for the Fra\"iss\'e limit of the class $\mathbb{K}^{P}_{c,\mathbb{Q}}$. We will use the fact that both $\mathbb{K}^{P}_{c,p}$ and $\mathbb{K}^{P}_{c,\mathbb{Q}}$ come with a notion of free amalgam: we will denote the free amalgam of $A$ and $B$ over $C$ by $A \otimes_{C} B$. 

\begin{thm}
\begin{enumerate}
    \item The reduct of $\mathbb{G}_{c,p}$ to the language of groups is the generic $c$-nilpotent group of exponent $p$. 
    \item The reduct of $\mathbb{G}_{c,\mathbb{Q}}$ to the language of $\mathbb{Q}$-powered groups is the generic $\mathbb{Q}$-powered $c$-nilpotent group.
    \item The reduct of $\mathbb{G}_{c,\mathbb{Q}}$ to the language of groups is the generic torsion-free $c$-nilpotent group. 
\end{enumerate}
\end{thm}

\begin{proof}
(1)  Let $\mathbf{G}$ be the underlying group of $\mathbb{G}_{c,p}$. Clearly any finitely generated subgroup of $\mathbf{G}$ is a group of exponent $p$ and nilpotency class $\leq c$. Conversely, any finitely generated group of exponent $p$ and nilpotency class $\leq c$ has an expansion to a structure in $\mathbb{K}_{c,p}^{P}$ by naming the lower central series, and therefore is isomorphic to a finitely generated substructure of $\mathbb{G}_{c,p}$. Therefore, $\mathrm{Age}(\mathbf{G}) = \mathbb{K}_{c,p}$. 

Next we show $\mathbf{G}$ is weak-$\mathbb{K}_{c,p}$-homogeneous. Fix $A \in \mathbb{K}_{c,p}$ and let $B \in \mathbb{K}_{c,p}$ be a UL-group with $A \subseteq B$.  Suppose $C \in \mathbb{K}_{c,p}$ and $j : B \to C$ is an embedding. Following \cite[Theorem 4.2.2]{kruckmanthesis2016}, let $\mathrm{WH}_{A,B,C}$ denote
the property of a group with underlying set $\omega$ with exponent $p$ and nilpotency class $\leq c$ that, if $A \subseteq B \subseteq M$, then there is an embedding $f : C \to M$ such that $f  \circ j  \circ i$ is the inclusion of $A$ in $M$.  We prove that Player II has a winning strategy in $G_{\mathbb{K}_{c,p}}(\mathrm{WH}_{A,B,C})$. Assuming Player $I$ has played $A_{k}$ (so $k$ is even), if $B \not\subseteq A_{k}$, then Player II plays arbitrarily. If $B \subseteq A_{k}$, then Player II plays $A_{k+1} := A_{k} \otimes_{B} C$, the free amalgam of $A_{k}$ and $C$ over $B$, where $A_{k}$, $B$, and $C$ are viewed as elements of $\mathbb{K}^{P}_{c,p}$ by interpreting the predicates as arbitrary Lazard series (which in the case of $B$ is canonical and preserved by the embeddings). It is clear that if $(A_{i})_{i < \omega}$ is a run of the game played according to this strategy and $M = \bigcup A_i$, then if $A \subseteq B \subseteq M$, then there is an embedding $f : C \to M$ such that $f  \circ j  \circ i$ is the inclusion of $A$ in $M$, since if $k$ is the least even number such that $B \subseteq A_{k}$, then we may map $C$ to its image in $A_{k+1} = A_{k} \otimes_{B} C$. 

Moreover, it is clearly generic that $A \subseteq B \subseteq M$, since if Player II wants to ensure that $A \subseteq B \subseteq M$, after Player I begins the game by playing $A_{0}$, Player II can respond, by JEP, by playing any structure $A_{1} \supseteq A_{0}$ into which $B$ embeds. This shows that $\mathbf{G}$ is $\mathbb{K}_{p,c}$-universal and weak-$\mathbb{K}_{p,c}$-homogeneous and therefore is the generic $c$-nilpotent group of exponent $p$. 

(2) The proof of (2) is identical to the proof of (1): we let $\mathbf{G}$ be the reduct of $\mathbb{G}_{c,\mathbb{Q}}$ to the underlying $\mathbb{Q}$-powered group, and then use the same arguments to show that $\mathrm{Age}(\mathbf{G}) = \mathbb{K}_{c,\mathbb{Q}}$ and that $\mathbf{G}$ is weak-$\mathbb{K}_{c,\mathbb{Q}}$-homogeneous. 

(3) Now let $\mathbf{G}$ denote the reduct of $\mathbb{G}_{c,\mathbb{Q}}$ to the language of groups (so the reduct of the $\mathbb{Q}$-powered group considered in (2), forgetting the $n$th root functions). First, we show $\mathrm{Age}(\mathbf{G}) = \mathbb{K}_{c,\mathrm{tf}}$. Fix any $A \in \mathbb{K}_{c,\mathrm{tf}}$ and let $\hat{A}$ be its Malcev completion. The group $\hat{A}$ has a unique expansion to a $\mathbb{Q}$-powered group which is an element of $\mathbb{K}_{c,\mathbb{Q}}$, hence embeds into $\mathbb{G}_{c,\mathbb{Q}}$ by (2). This shows $A \in \mathrm{Age}(\mathbf{G})$. Clearly each finitely generated subgroup of $\mathbf{G}$ is at most $c$-nilpotent and torsion free, hence $\mathrm{Age}(\mathbf{G}) = \mathbb{K}_{c,\mathrm{tf}}$. 

Next, we show that $\mathbf{G}$ is weak-$\mathbb{K}_{c,tf}$-homogeneous. Pick any $A \in \mathbb{K}_{c,\mathrm{tf}}$. Choose any $B \in \mathbb{K}_{c,tf}$ with $A \subseteq B$ such that $\hat{B}$ is a UL-group (take, for example, $B$ to be generated by a finite generating set of $A$, as a group, together with a finite set of generators for $B$ as a $\mathbb{Q}$-powered group). Suppose $C \in \mathbb{K}_{c,\mathrm{tf}}$ and $j : B \to C$ is an embedding. Now, as in (1), let $\mathrm{WH}_{A,B,C}$ denote
the property of a torsion-free group with underlying set $\omega$ and nilpotency class $\leq c$ that, if $A \subseteq B \subseteq M$, then there is an embedding $f : C \to M$ such that $f  \circ j  \circ i$ is the inclusion of $A$ in $M$. Now we adapt the strategy from (1) to a winning strategy for Player II in the game $G_{\mathbb{K}_{c,\mathrm{tf}}}(\mathrm{WH}_{A,B,C})$.   Assuming Player $I$ has played $A_{k}$ (so $k$ is even), if $B \not\subseteq A_{k}$, then Player II plays arbitrarily. If $B \subseteq A_{k}$, then Player II plays any $A_{k+1}$ containing $A_{k}$ and containing a finite generating set for the free amalgam $\hat{A_{k}} \otimes_{\hat{B}} \hat{C}$, the free amalgam of the Malcev completions of $A_{k}$ and of $C$ over the Malcev completion of $B$, viewed as $\mathbb{Q}$-powered groups in $\mathbb{K}^{P}_{c,\mathbb{Q}}$, where the predicates are interpreted by arbitrary Lazard series (which, in the case of $\hat{B}$ is canonical). As above, it is clear that if $(A_{i})_{i < \omega}$ is a run of the game played according to this strategy and $M = \bigcup A_i$, then if $A \subseteq B \subseteq M$, then there is an embedding $f : C \to M$ such that $f  \circ j  \circ i$ is the inclusion of $A$ in $M$, since if $k$ is the least even number such that $B \subseteq A_{k}$, then we may map $C$ to its image in $A_{k+1} = A_{k} \otimes_{B} C$, and it is generic that $A \subseteq B \subseteq M$, which shows that $\mathbf{G}$ is $\mathbb{K}_{c,\mathbb{Q}}$-universal and weak-$\mathbb{K}_{c,\mathbb{Q}}$-homogeneous and therefore is the generic $c$-nilpotent torsion-free group.

\end{proof}

\begin{cor}
    The generic torsion-free $c$-nilpotent group is $\mathbb{Q}$-powered. 
\end{cor}

\bibliographystyle{plain}
\bibliography{biblio.bib}{}

\end{document}

%% file: arxiv_preamble.tex
\usepackage{enumerate}
\usepackage{url}
\usepackage[font=small,labelfont=bf]{caption}
\usepackage[all,arc]{xy}
\usepackage{tabularx}
\usepackage{adjustbox}

\usepackage[normalem]{ulem}

\usepackage{amssymb,latexsym}
\usepackage{amsmath,amsthm}
\usepackage{amsfonts,mathrsfs}
\usepackage{mathtools}

\usepackage{tikz}
\usetikzlibrary{shapes.geometric, arrows, positioning,decorations.pathreplacing,calligraphy}

\tikzstyle{bloc} = [rectangle, rounded corners, 
minimum width=3cm, 
minimum height=1cm,
text centered, 
draw=black, 
fill=airforceblue!30]

\tikzstyle{decision} = [diamond,
minimum width=3cm, 
minimum height=1cm, 
text centered, 
draw=black, 
fill=airforceblue!30]
\tikzstyle{arrow} = [thick,->,>=stealth]

\tikzstyle{io} = [trapezium, 
trapezium stretches=true, 
trapezium left angle=70, 
trapezium right angle=110, 
minimum width=3cm, 
minimum height=1cm, text centered, 
draw=black, fill=blue!30]

\tikzstyle{process} = [rectangle, 
minimum width=3cm, 
minimum height=1cm, 
text centered, 
text width=3cm, 
draw=black, 
fill=orange!30]

\usepackage{tikz-cd}
\usepackage{todonotes}

\usepackage[all]{xy}
\usepackage{graphicx}

\usepackage{hyperref}
\hypersetup{
    colorlinks,
    citecolor=blue,
    filecolor=blue,
    linkcolor=blue,
    urlcolor=blue
}

\usepackage{enumitem}

\usepackage{listings}
\usepackage{color}

\definecolor{dkgreen}{rgb}{0,0.6,0}
\definecolor{gray}{rgb}{0.5,0.5,0.5}
\definecolor{mauve}{rgb}{0.58,0,0.82}

\makeatletter
\newcommand{\ast@scaled}[2]{\raisebox{#2}{\scalebox{#1}{$\ast$}}}
\newcommand{\bigast}{%
  \mathop{%
    \mathchoice
      {\ast@scaled{1.7}{-0.40ex}}
      {\ast@scaled{1.2}{-0.13ex}}
      {\ast@scaled{1.0}{0ex}}
      {\ast@scaled{0.8}{0.06ex}}
  }}
\makeatother

\newtheorem{thm}{Theorem}[section]
\newtheorem{cor}[thm]{Corollary}
\newtheorem{prop}[thm]{Proposition}

\newtheorem{lem}[thm]{Lemma}

\newtheorem{theorem}[thm]{Theorem}

\newtheorem*{theorem*}{Theorem}

\theoremstyle{definition}
\newtheorem{defn}[thm]{Definition}

\newtheorem{fact}[thm]{Fact}

\theoremstyle{definition}
\newtheorem{definition}[thm]{Definition}

\newtheorem{question}[thm]{Question}

\theoremstyle{remark}

\makeatletter
\let\c@equation\c@thm
\makeatother
\numberwithin{equation}{section}

\newcommand\Q{\mathbb{Q}}
\newcommand\N{\mathbb{N}}

\newcommand\F{\mathbb{F}}

\newcommand\CC{\mathscr{C}}

\newcommand{\lev}{\mathrm{lev}}

\def\seq{\subseteq}

\newcommand{\vect}[1]{\langle {#1} \rangle}

\DeclareMathOperator{\Span}{span}

\definecolor{airforceblue}{rgb}{0.36, 0.54, 0.66}

\def\Ind{\setbox0=\hbox{$x$}\kern\wd0\hbox to 0pt{\hss$\mid$\hss}
\lower.9\ht0\hbox to 0pt{\hss$\smile$\hss}\kern\wd0}
\def\Notind{\setbox0=\hbox{$x$}\kern\wd0\hbox to 0pt{\mathchardef
\nn=12854\hss$\nn$\kern1.4\wd0\hss}\hbox to
0pt{\hss$\mid$\hss}\lower.9\ht0 \hbox to 0pt{\hss$\smile$\hss}\kern\wd0}

\def\indi#1{\mathop{\ \ \hbox to 0ex{\hss$\vert^{\hbox to 0ex{$\scriptstyle#1$\hss}}$\hss}
\lower1ex\hbox to 0ex{\hss$\smile$\hss}\ \ }}

\def\nindi#1{\mathop{\ \ \hbox to 0ex{\hss$\!\not{\vert}^{\hbox to 0ex{$\scriptstyle\,#1$\hss}}$\hss}
\lower1ex\hbox to 0ex{\hss$\smile$\hss}\ \ }}

